\documentclass[letterpaper, 10 pt, conference]{ieeeconf}
\usepackage[utf8]{inputenc}

\IEEEoverridecommandlockouts

\usepackage{balance}

\usepackage{multirow}

\usepackage{amsmath,mathtools} %
\usepackage{amsfonts,amssymb}
\usepackage{bbm}

\usepackage{xspace}  %

\usepackage{graphicx}
\graphicspath{{figs/}}

\usepackage{subcaption}
\usepackage{caption}

\usepackage{hyperref}
\hypersetup{colorlinks=true, linkcolor= blue, allcolors=blue, citecolor = blue, filecolor=black, urlcolor=blue}

\usepackage{todonotes}
\usepackage{comment}

\usepackage{tikz}
\usetikzlibrary{arrows,positioning}  %

\usepackage{pgfplots}\pgfplotsset{compat=1.9}
\usepackage{grffile}
\pgfplotsset{compat=newest} %
\usetikzlibrary{plotmarks}
\usetikzlibrary{arrows.meta}
\usepgfplotslibrary{patchplots}
\usepgfplotslibrary{groupplots}

\pgfplotsset{compat=1.15}
\usepackage{mathrsfs}
\usetikzlibrary{arrows}

\usepackage{tabularx}

\usepackage[ruled]{algorithm}
\usepackage[noend]{algpseudocode}

\usepackage{diagbox}
\newcolumntype{K}[1]{>{\centering\arraybackslash}p{#1}}

\makeatletter
\newcommand{\multiline}[1]{%
 \begin{tabularx}{\dimexpr\linewidth-\ALG@thistlm}[t]{@{}X@{}}
  #1
 \end{tabularx}
}
\makeatother

\usepackage[per-mode=symbol]{siunitx}

\usepackage{cite} %
\usepackage{array} %
\usepackage{booktabs}      %

\newtheorem{theorem}{Theorem}

\newtheorem{proposition}[theorem]{Proposition}

\newcounter{remark}
{\par\endtrivlist\unskip}

\newcounter{problem}
{\par\endtrivlist\unskip}

\usepackage{mylatexdefs}

\newcommand{\bbsym}[1]{\ensuremath{\boldsymbol{#1}}}
\newcommand{\NSL}{N_{\mathrm{SL}}}
\newcommand{\NSSL}{N_{\mathrm{SSL}}}

\usepackage{stfloats}%

\begin{document}

\title{\LARGE \bf Learning to Solve Mixed-Integer Nonlinear Programming with Neural Networks and Sequential Quadratic Programming Refinement}
\author{Viet-Anh Le, {\IEEEmembership{Member, IEEE}},
and Rahul Mangharam, {\IEEEmembership{Senior Member, IEEE}} 
\thanks{The authors are with the Department of Electrical \& Systems Engineering, University of Pennsylvania, Philadelphia, PA 19104, USA (emails: {\tt\small \{vietanh,rahulm\}@engineering.upenn.edu}).}
\thanks{This work was partially supported by US DoT Safety21 National University Transportation Center and NSF under Grant CISE-$2431569$.}
}
\maketitle

\begin{abstract}
We propose a learning-to-optimize framework for solving parametric mixed-integer nonlinear programs arising in model predictive control and trajectory optimization. 
The proposed method combines neural network prediction with a sequential quadratic programming (SQP) refinement layer. 
First, neural networks are trained using a combination of supervised and self-supervised losses to predict integer solutions together with compatible continuous warm starts. 
To account for prediction inaccuracies, the trained predictor generates multiple candidate integer solutions, each paired with a continuous initialization. 
The integer solutions are then fixed, and a differentiable SQP layer refines the associated continuous variables over a prescribed number of iterations. 
The SQP-layer parameters, including the curvature regularization and step size, are trained to improve both feasibility and objective value after every iteration.
By separating integer prediction from continuous optimization-based refinement, the proposed framework improves robustness to imperfect neural predictions while maintaining a predictable computational budget. 
We demonstrate the proposed framework on contact and trajectory optimization for a planar box-pushing example.
\end{abstract}

\section{Introduction}

Mixed-integer optimization arises naturally in many control applications involving coupled continuous control and discrete decision-making, including autonomous driving \cite{quirynen2024real,le2024distributed}, robotics \cite{acosta2025perceptive,le2025multirobot}, and hybrid systems \cite{marcucci2020warm}, \cite{mcallister2022advances}.
Despite the advantages offered by mixed-integer optimization for trajectory optimization and model predictive control (MPC), its real-time implementation remains challenging because solving the mixed-integer programs combines the combinatorial complexity of discrete optimization with the computational burden of continuous optimization.
Learning-to-optimize (L2O) provides a promising approach for reducing the computational cost required for online decision-making by replacing or accelerating iterative solvers with learned mappings from problem parameters to solutions.

In recent years, there has been growing research interest in L2O methods for mixed-integer programming.
Classical supervised-learning approaches for predicting integer solutions in mixed-integer optimization have been proposed in \cite{cauligi2020learning,cauligi2021coco,le2025multirobot}. 
These methods primarily focus on mixed-integer convex programs, where fixing the learned integer variables reduces the original problem to convex continuous optimization problems.
In \cite{xu2025mixed}, a hybrid-action reinforcement learning approach was proposed to jointly learn discrete and continuous control decisions for mixed-integer optimal control. 
However, both supervised learning and reinforcement learning typically require substantial training data, and the resulting models may still produce infeasible solutions.
Self-supervised learning approaches \cite{tang2024learning,boldocky2025learning} have been developed for solving mixed-integer nonlinear programming (MINLPs) by training neural networks (NNs) with loss functions that combine the original objective with penalty terms for constraint violations. 
Though self-supervised learning can improve feasibility by penalizing constraint violations during training, such approaches may yield highly suboptimal solutions, since gradient-based training can converge to poor local optima. 
In \cite{le2026hybrid}, we developed a hybrid L2O framework for mixed-integer quadratic programming (MIQP) that combines neural prediction of the integer variables with QP-based continuous refinement and is trained using a combination of supervised and self-supervised loss functions.

Though NNs trained with either supervised or self-supervised losses can approximate MINLP solutions, prediction errors may still be present and lead to suboptimal or infeasible solutions. 
To improve reliability, an online correction mechanism is therefore necessary. 
For example, \cite{cauligi2021coco} considers evaluating multiple high-probability discrete strategies predicted by the NNs and solving the corresponding optimization problems sequentially until a feasible solution is found.
However, this approach is primarily suited to mixed-integer convex programs, where fixing the discrete strategy yields a convex continuous optimization problem. 
In this paper, we propose an L2O framework for MINLPs using NNs to generate multiple solution candidates and a sequential quadratic programming (SQP) refinement layer. 
The NNs are trained to learn both optimal integer decisions and continuous initializations. 
Each candidate is then refined by a differentiable elastic trust-region SQP layer while keeping the integer decisions fixed, with trainable SQP parameters optimized to improve the continuous solution within a few iterations. 
During online inference, given a fixed number of candidate integer solutions and their corresponding continuous solutions refined through SQP iterations, the final solution is selected based on feasibility and objective value.
Thus, this design incorporates the constrained optimization structure of the original problem while maintaining a predictable online computational budget. 
We demonstrate the performance of the proposed framework on a contact and trajectory optimization problem for planar box pushing.

Compared to the most relevant work in the literature \cite{tang2024learning}, which also develops an L2O framework for MINLP, this paper differs in several aspects.
First, rather than predicting relaxed MINLP solutions with an NN, we train the NN to directly predict mixed-integer solutions that are subsequently refined by the optimization layer. 
Second, while the correction layers in \cite{tang2024learning} are designed primarily to enforce the integrality of relaxed MINLP solutions, our refinement layer is trained to further improve the predicted mixed-integer candidates online with respect to feasibility and optimality. 
Finally, by incorporating an SQP as a refinement layer, our framework enables online optimization to be more explicitly embedded within the L2O architecture.

The remainder of this paper is organized as follows. 
Section~\ref{sec:prb} presents the formulation of the parametric MINLP problem considered in this work. Section~\ref{sec:frwrk} introduces the proposed architecture that combines NN prediction with differentiable SQP refinement, while
Section~\ref{sec:train} presents the two-phase training framework.
Section~\ref{sec:sim} evaluates the proposed approach on a contact-scheduling and trajectory optimization problem for planar box pushing. 
Finally, Section~\ref{sec:conc} concludes the paper.

\section{Problem Formulation}
\label{sec:prb}

We consider a class of parametric mixed-integer optimal control problems (MIOCPs) arising in MPC or trajectory optimization, where the problem parameters may include, for example, the initial state, target state, and other task- or environment-dependent quantities. 
Let $\bbsym{x}_t \in \XXX \subset \RR^{n_x}$ be the system state, $\bbsym{u}_t \in \UUU \subset \RR^{n_u}$ be the continuous control inputs and $\bbsym{z}_t \in \ZZZ \subset \ZZ^{n_z}$ be the integer decision variables. 
Let $H$ be the control horizon length, and $\bbsym{x}_{0:H}, \bbsym{u}_{0:H-1}, \bbsym{z}_{0:H-1}$ be the concatenated vectors for each variable over the control horizon.
Given a vector of problem parameters ${\bbsym{\theta} \in \Theta \subset \RR^{n_p}}$, a parametric MIOCP, where the objective function and constraints are functions of $\bbsym{\theta}$, can be generally written as:
\begin{equation}
\label{eq:miocp}
\begin{aligned}
\minimize_{\bbsym{x}_{0:H}, \bbsym{u}_{0:H-1}, \bbsym{z}_{0:H-1}} \; & c_H(\bbsym{x}_H; \bbsym{\theta}) + \sum_{t=0}^{H-1}\; c_t(\bbsym{x}_t, \bbsym{u}_t, \bbsym{z}_t; \bbsym{\theta}), \\
\text{subject to \quad} \; 
& \bbsym{x}_0 = \bbsym{x}_{\text{init}}(\bbsym{\theta}), \\
& \bbsym{x}_{t+1} = \bbsym{f} (\bbsym{x}_t, \bbsym{u}_t, \bbsym{z}_t; \bbsym{\theta}), \\
& \bbsym{g}_{t}(\bbsym{x}_t, \bbsym{u}_t, \bbsym{z}_t; \bbsym{\theta}) \le 0,\\
& \bbsym{g}_{H} (\bbsym{x}_H; \bbsym{\theta}) \le 0,
\end{aligned}
\end{equation}
with $t = 0, \ldots, H-1$, where the stage cost $c_t(\cdot)$, the terminal cost $c_H (\cdot)$, the dynamics $\bbsym{f}(\cdot)$, the inequality constraints $\bbsym{g}_{t}(\cdot)$ and $\bbsym{g}_{H}(\cdot)$ can be nonlinear, but are assumed to be differentiable.

The MIOCP \eqref{eq:miocp} can be represented by a parametric MINLP.
The continuous decision variable of the MINLP is defined as the stacked state-control trajectory as follows:
\begin{equation}
\bbsym x = \bigl[ \bbsym x_0^\top, \bbsym u_0^\top, \cdots, \bbsym u_{H-1}^\top, \bbsym x_H^\top \bigr]^{\top}.
\end{equation}
Similarly, the discrete decision variables over the prediction horizon are stacked as:
\begin{equation}
\bbsym z =
\begin{bmatrix}
\bbsym z_0^\top, \cdots, \bbsym z_{H-1}^\top
\end{bmatrix}^{\top}
\in \ZZZ,
\end{equation}
where $\ZZZ$ denotes the appropriate set for the integer variables.
The parametric MINLP can be represented in the following compact form:
\begin{equation}
\label{eq:para-minlp}
\begin{aligned}
\minimize_{\bbsym x, \bbsym z}\quad & f(\bbsym x, \bbsym z; \bbsym \theta), \\
\subjectto \quad & \bbsym h(\bbsym x, \bbsym z; \bbsym \theta) = 0, \\
& \bbsym g(\bbsym x, \bbsym z; \bbsym \theta) \le 0, \\
& \bbsym z \in \ZZZ.
\end{aligned}  
\end{equation}

L2O techniques can be used to approximate the solution map of \eqref{eq:para-minlp}.
However, directly predicting the solution using NNs may yield infeasible or highly suboptimal results, whereas conventional iterative MINLP algorithms, such as mixed-integer SQP \cite{quirynen2021sequential,ghezzi2024sequential}, are difficult to execute effectively within a fixed number of iterations and often require substantial problem-specific tuning.
To address these challenges, we decompose the learning process into two stages: NNs first predict multiple candidate integer solutions together with their associated continuous warm starts, after which a fixed number of differentiable elastic trust-region SQP layers refine the continuous variables. 

\begin{figure*}
\centering
\includegraphics[width=0.98\textwidth,trim=0 10 0 5, clip]{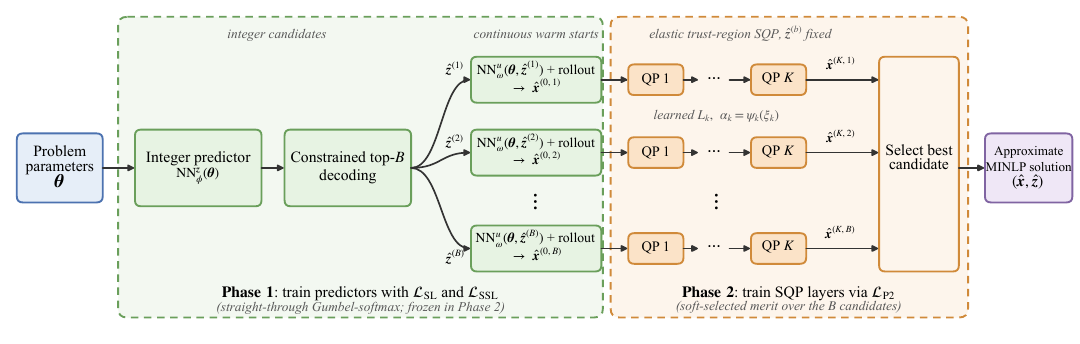}
\caption{Overview of the proposed framework. The architecture consists of NN prediction for candidate integer solutions and continuous warm starts, followed by differentiable SQP refinement and final candidate selection.}
\label{fig:overview}
\vspace{-6mm}
\end{figure*}

\section{Proposed Architecture}
\label{sec:frwrk}

In this section, we describe the proposed framework for learning how to solve a parametric MINLP using two coupled components:
(1) neural networks predict the discrete decisions and an initial continuous solution, and (2) a fixed number of trainable SQP iterations refine the continuous solution.
An overview of the proposed architecture is shown in Figure~\ref{fig:overview}.

\subsection{Neural Solution Map and SQP Refinement Layer}

Given the problem parameters $\bbsym{\theta}$, we first employ an NN to predict the optimal integer decision and an associated continuous solution that serves as a warm start. 
The integer predictor first produces a distribution over admissible integer solutions,
\begin{equation}
\bbsym{\ell}(\bbsym{\theta}) = \mathrm{NN}^{z}_{\bbsym{\phi}}(\bbsym{\theta}),
\end{equation}
where $\bbsym{\phi}$ denotes the NN weights and $\bbsym{\ell}$ represents the corresponding logits. 
For a candidate integer solution $\hat{\bbsym z} (\bbsym{\theta})$ obtained from these logits, the continuous predictor generates a compatible initial control sequence as
\begin{equation}
\hat{\bbsym u}^{(0)} = \mathrm{NN}^{u}_{\bbsym{\omega}} \left( \bbsym{\theta}, \hat{\bbsym z} (\bbsym{\theta}) \right),
\end{equation}
where $\bbsym{\omega}$ denotes the NN weights of the continuous predictor.
In other words, the continuous warm start is conditioned on the predicted integer solution.
In addition, we obtain the state trajectory by forward rollout of the system dynamics using the predicted control sequence $\hat{\bbsym u}^{(0)}(\bbsym \theta)$ and possibly the discrete decisions $\hat{\bbsym z}(\bbsym \theta)$. 
Specifically,
\begin{equation}
\bbsym x_{t+1}^{(0)} = \bbsym f\left( \bbsym x_t^{(0)}, \hat{\bbsym u}_{t}^{(0)},
\hat{\bbsym z}_{t}; \bbsym\theta \right), \; t=0,\ldots,H-1,
\end{equation}
with $\bbsym x_0^{(0)} = \bbsym x_{\mathrm{init}}(\bbsym\theta)$.
The resulting state-control trajectory is denoted by $\hat{\bbsym x}^{(0)}(\bbsym\theta)$ and is used as an initialization for the SQP layer, while the predicted discrete decisions $\hat{\bbsym z} (\bbsym{\theta})$ are fixed throughout the SQP iterations.
Since the state trajectory is generated by forward rollout, $\hat{\bbsym x}^{(0)}(\bbsym\theta)$ satisfies the system dynamics by construction. 
After fixing $\hat{\bbsym z} (\bbsym \theta)$, the remaining problem is a continuous NLP:
\begin{equation}
\label{eq:fi-nlp}
\begin{aligned}
\minimize_{\bbsym x} \quad & f \bigl(\bbsym x, \hat{\bbsym z} (\bbsym \theta); \bbsym \theta \bigr), \\
\subjectto \quad
& \bbsym h \bigl(\bbsym x,\hat{\bbsym z} (\bbsym \theta);\bbsym \theta \bigr)=0, \\
& \bbsym g \bigl(\bbsym x, \hat{\bbsym z} (\bbsym \theta); \bbsym \theta \bigr)\le 0.
\end{aligned}  
\end{equation}
The SQP refinement layer approximately solves this fixed-integer NLP. 
Starting from $\hat{\bbsym x}^0(\bbsym\theta)$, we apply $K$ QP iterations within the SQP refinement layer.
In particular, the SQP refinement layer generates a sequence of continuous iterates $\hat{\bbsym x}^1, \ldots, \hat{\bbsym x}^{(K)}$ according to
\begin{equation}
\hat{\bbsym x}^{(k+1)} = \hat{\bbsym x}^{(k)} + \alpha_k \bbsym d^{(k)},
\; k=0,\ldots,K-1,
\end{equation}
where $\alpha_k > 0$ is a step length and $\bbsym d^{(k)}$ is obtained by solving a QP approximation of \eqref{eq:fi-nlp} around the current iteration $\hat{\bbsym x}^{(k)}$.
For each iteration $k$, we define
\begin{equation}
\begin{split}
f_k &= f \bigl( \hat{\bbsym x}^{(k)}, \hat{\bbsym z}(\bbsym \theta);
\bbsym \theta \bigr), \;
\bbsym h_k = \bbsym h\bigl( \hat{\bbsym x}^{(k)}, \hat{\bbsym z}(\bbsym \theta); \bbsym \theta \bigr),
\\
\bbsym g_k &= \bbsym g\bigl( \hat{\bbsym x}^{(k)}, \hat{\bbsym z}(\bbsym \theta); \bbsym \theta \bigr),
\end{split}
\end{equation}
and
\begin{equation}
\begin{split}
\bbsym c_k &= \nabla_{\bbsym x}
f\bigl( \hat{\bbsym x}^{(k)}, \hat{\bbsym z}(\bbsym \theta); \bbsym \theta \bigr), \,
\bbsym A_k = \nabla_{\bbsym x} \bbsym h \bigl( \hat{\bbsym x}^{(k)}, \hat{\bbsym z}(\bbsym \theta); \bbsym \theta \bigr), \\
\bbsym C_k &= \nabla_{\bbsym x} \bbsym g\bigl( \hat{\bbsym x}^{(k)}, \hat{\bbsym z}(\bbsym \theta); \bbsym \theta \bigr).
\end{split}
\end{equation}
Moreover, let $\bbsym B_{\mathrm{quad}}$ denote the exact Hessian of the terms in the objective function that are convex quadratic in the continuous decision variables.
These terms are retained exactly in the quadratic objective of the QP subproblem. 
The remaining second-order information associated with the nonlinear terms is captured by a learned curvature matrix $\bbsym B_{\mathrm{nl},k}$. 
The total curvature matrix in the $k$-th QP subproblem is therefore computed as:
\begin{equation}
\bbsym B_k = \bbsym B_{\mathrm{quad}} + \bbsym B_{\mathrm{nl},k}.
\end{equation}
In a classical SQP method, the curvature associated with the nonlinear terms may be obtained from the exact Hessian of the Lagrangian or from a quasi-Newton approximation. 
In the proposed SQP refinement, $\bbsym B_{\mathrm{nl},k}$ is parameterized and learned while ensuring that $\bbsym B_k$ remains positive definite.

The QP approximation of \eqref{eq:fi-nlp} around the current iteration $\hat{\bbsym x}^{(k)}$ is given by:
\begin{equation}
\label{eq:qp-subproblem}
\begin{aligned}
\minimize_{\bbsym d} \quad &
\frac{1}{2} \bbsym d^\top \bbsym B_k \bbsym d + \bbsym c_k^\top \bbsym d, \\
\subjectto \quad & \bbsym A_k\bbsym d+\bbsym h_k=\bbsym 0, \\
& \bbsym C_k\bbsym d+\bbsym g_k\leq\bbsym 0, \\
& -\Delta_k\bbsym 1 \leq \bbsym d \leq \Delta_k\bbsym 1,
\end{aligned}
\end{equation}
where $\Delta_k>0$ denotes the trust-region radius used to restrict the QP approximation to a local neighborhood.
Since the QP approximation \eqref{eq:qp-subproblem} may become infeasible, even if the fixed-integer nonlinear program \eqref{eq:fi-nlp} itself is feasible, due to local linearization errors, we introduce slack variables to guarantee feasibility, leading to the following elastic QP problem:
\begin{equation}
\label{eq:elastic-qp-subproblem}
\begin{aligned}
\minimize_{\bbsym d, \bbsym s_h, \bbsym s_g} \quad
& \frac{1}{2} \bbsym d^\top \bbsym B_k \bbsym d + \bbsym c_k^\top \bbsym d
+ \rho_{h,k} \norm{\bbsym s_h}_2^2 + \rho_{g,k} \norm{\bbsym s_g}_2^2,
\\
\subjectto \quad & \bbsym A_k \bbsym d + \bbsym h_k = \bbsym s_h, \\
& \bbsym C_k \bbsym d + \bbsym g_k \le \bbsym s_g, \\
& \bbsym s_g \ge \bbsym 0, \\
& - \Delta_k \bbsym 1 \le \bbsym d \le \Delta_k \bbsym 1,
\end{aligned}  
\end{equation}
where $\bbsym s_h$ and $\bbsym s_g$ are the slack variables corresponding to the equality and inequality constraints, while $\rho_{h,k}$ and $\rho_{g,k}$ are the slack penalty weights.
The solution of \eqref{eq:elastic-qp-subproblem} is written compactly as
$\bbsym d^{(k)} = \QQQ_k \left( \hat{\bbsym x}^{(k)}; \bbsym \theta, \hat{\bbsym z}, \bbsym \psi_k \right)$,
where $\bbsym \psi_k$ collects the trainable parameters of the $k$-th QP layer, which will be described in the next subsection.
The corresponding SQP update is:
\begin{equation}
\label{eq:sqp_update_trainable}
\hat{\bbsym x}^{(k+1)} = \hat{\bbsym x}^{(k)} + \alpha_k \mathcal Q_k \left( \hat{\bbsym x}^{(k)}; \bbsym \theta, \hat{\bbsym z}, \bbsym \psi_k \right).
\end{equation}

In summary, the final continuous solution after $K$ QP iterations is denoted by
\begin{equation}
\hat{\bbsym x} (\bbsym\theta) = \bbsym T_{\bbsym\psi}^{K}
\left( \hat{\bbsym x}^{(0)}(\bbsym\theta), \hat{\bbsym z}(\bbsym\theta); \bbsym\theta \right),
\end{equation}
where $\bbsym \psi = \left\{ \bbsym \psi_0,\ldots,\bbsym \psi_{K-1} \right\}$.

\subsection{Trainable and Differentiable QP Layers}
\label{subsec:differentiable_qp}

If $\bbsym B_k \succ 0$, the elastic trust-region SQP subproblem is strongly convex with respect to the SQP direction $\bbsym d$. 
Moreover, for any fixed $\bbsym d$, the equality slack is uniquely determined by $\bbsym s_h=\bbsym A_k\bbsym d+\bbsym h_k$, while the optimal
inequality slack is $\bbsym s_g =[\bbsym C_k\bbsym d+\bbsym g_k]_+$.
Consequently, the QP subproblem admits a unique primal optimizer, if it exists.
Under standard regularity conditions and a locally fixed active set, the parametric solution map is locally differentiable, and its derivatives can be obtained by implicitly differentiating the KKT conditions \cite{amos2017optnet,agrawal2019differentiable,magoon2026differentiation}. 
Consequently, a strongly convex QP can be integrated as a differentiable optimization layer and the gradients of the training loss are propagated through both the QP solution maps and the SQP updates across all $K$ layers, enabling optimization of the trainable SQP parameters.
Given a training loss $\LLL$, the gradient with respect to the parameters of the $k$-th QP layer is
\begin{equation}
\label{eq:layer_gradient}
\frac{\partial \LLL}{\partial \bbsym \psi_k}
=
\frac{\partial \LLL}{\partial \bbsym x^{(K)}}
\left(
\prod_{l=k+1}^{K-1}
\frac{\partial \bbsym x^{(l+1)}}
{\partial \bbsym x^{(l)}}
\right)
\frac{\partial \bbsym x^{(k+1)}}
{\partial \bbsym \psi_k}.
\end{equation}
The term
$\partial \bbsym x^{(k+1)}/\partial \bbsym \psi_k$ contains the derivative of the QP solution map:
\begin{equation}
\frac{\partial \bbsym x^{(k+1)}}{\partial \bbsym \psi_k}
= \bbsym d^{(k)} \frac{\partial \alpha_k}{\partial \bbsym \psi_k}
+ \alpha_k \frac{\partial \bbsym d^{(k)}}{\partial \bbsym \psi_k},
\end{equation}
where $\partial \bbsym d^{(k)}/\partial \bbsym \psi_k$ is obtained by differentiating through the QP layer.
The details of differentiating through a convex QP layer can be found in
\cite{amos2017optnet,agrawal2019differentiable,magoon2026differentiation}.
Next, we describe the trainable SQP parameters at the $k$-th QP layer.
In this paper, we consider that only the curvature and the step-length parameter are trainable, while slack penalty weights and the trust-region radius are fixed across the layers.

\paragraph{Trainable curvature approximation}
For the terms in the objective with an exact quadratic structure, the corresponding curvature
$\bbsym{B}_{\mathrm{quad}}$ is retained explicitly. 
For the remaining nonlinear terms, we parameterize the curvature approximation as:
\begin{equation}
\label{eq:trainable_Bnl}
\bbsym{B}_{\mathrm{nl},k} = \bbsym{L}_k\bbsym{L}_k^\top + \mu_{\min}\bbsym{I},
\; \mu_{\min}>0.
\end{equation}
The total curvature matrix used in the $k$-th SQP subproblem is therefore
\begin{equation}
\label{eq:trainable_B}
\bbsym{B}_k = \bbsym{B}_{\mathrm{quad}} + \bbsym{L}_k\bbsym{L}_k^\top + \mu_{\min}\bbsym{I}.
\end{equation}
Since $\bbsym{B}_{\mathrm{quad}} \succeq 0$, this construction guarantees $\bbsym{B}_k\succ0$, yielding a strictly convex quadratic model with respect to the SQP search direction. 
For large-scale problems, $\bbsym{L}_k$ may be restricted to a diagonal or low-rank parameterization to reduce the number of trainable parameters.

\paragraph{Trainable step length}
The SQP step length is parameterized as
\begin{equation}
\alpha_k
=
\alpha_{\max}\operatorname{sigmoid}(a_k)
=
\frac{\alpha_{\max}}{1+\exp(-a_k)},
\end{equation}
which guarantees
$0<\alpha_k<\alpha_{\max}$.

We use a lightweight neural network $\bbsym{\psi}_k$ to adapt the local quadratic model and step length to the current problem instance and iteration as follows:
\begin{equation}
\left( \bbsym{L}_k,a_k \right) = \bbsym{\psi}_k(\bbsym{\xi}_k),
\end{equation}
where $\bbsym{\xi}_k$ denotes a feature vector characterizing the current problem instance and SQP iteration.

In the next proposition, we show that, under the proposed parameterization, the elastic trust-region QP at each SQP iteration is strongly convex with respect to $\bbsym d$ and admits a unique primal solution. 
Under standard regularity conditions, its solution map is locally differentiable with respect to the QP parameters, thereby enabling well-posed training.

\begin{proposition}
\label{prop:sqp_layer_differentiability}
\textit{(Well-posedness and local differentiability of an SQP layer)}
Consider the elastic trust-region QP at the $k$-th SQP layer \eqref{eq:elastic-qp-subproblem} given the parameterization of $\bbsym B_k$ and $\alpha_k$.
Let $\bbsym p_k$ denote the collection of parameters defining the elastic trust-region QP, \ie $\bbsym p_k = \left(\bbsym B_k, \bbsym c_k, \bbsym A_k, \bbsym h_k, \bbsym  C_k,\bbsym g_k\right)$.
Then the following statements hold.

\begin{enumerate}
\item The elastic QP \eqref{eq:elastic-qp-subproblem} is feasible for any finite $\bbsym h_k$ and $\bbsym g_k$ and is strongly convex in the primal variables $(\bbsym d,\bbsym s_h,\bbsym s_g)$.
Consequently, it admits a unique primal optimal solution $(\bbsym d^{(k)},\bbsym s_h^k,\bbsym s_g^k)$.

\item If, at a nominal parameter $\bar{\bbsym p}_k$, the linear independence constraint qualification (LICQ) and strict complementarity hold for the active constraints, then the local primal solution map $\bbsym p_k \mapsto (\bbsym d^{(k)},\bbsym s_h^{(k)},\bbsym s_g^{(k)})$
is continuously differentiable in a neighborhood of $\bar{\bbsym p}_k$.
\end{enumerate}

Consequently, each SQP layer is locally differentiable with respect to its inputs and parameters.
\end{proposition}

\begin{proof}
Feasibility follows by choosing $\bbsym d=\bbsym 0$, $\bbsym s_h=\bbsym h_k$, and $\bbsym s_g=[\bbsym g_k]_+$, which satisfies all constraints of \eqref{eq:elastic-qp-subproblem}.

Next, let $\bbsym v = [\bbsym d^\top, \bbsym s_h^\top, \bbsym s_g^\top]^{\top}$ denote the stacked primal variable of the elastic QP.
The Hessian of the objective with respect to $\bbsym v$ is
\begin{equation}
\nabla_{\bbsym v\bbsym v}^{2} q_k =
\begin{bmatrix}
\bbsym B_k & \bbsym 0 & \bbsym 0 \\
\bbsym 0 & \rho_{h,k}\bbsym I & \bbsym 0 \\
\bbsym 0 & \bbsym 0 & \rho_{g,k}\bbsym I
\end{bmatrix}.
\end{equation}
By design of the parameterization, $\bbsym B_k\succ0$, $\rho_{h,k}>0$, and $\rho_{g,k}>0$; therefore, $\nabla_{\bbsym v\bbsym v}^{2} q_k\succ0$.
Hence, the elastic QP objective is strongly convex in the primal variable $\bbsym v$, and the feasible elastic QP admits a
unique primal optimal solution $(\bbsym d^{(k)}, \bbsym s_h^{(k)}, \bbsym s_g^{(k)})$.

Next, we prove local differentiability. 
For a fixed active set, the KKT conditions of \eqref{eq:elastic-qp-subproblem} can be written compactly as $ F(\bbsym w^{(k)}, \bbsym p_k)=\bbsym 0, $ where $\bbsym w^{(k)}$ contains the primal and dual variables of the QP. 
The mapping $F$ is continuously differentiable with respect to both $\bbsym w^{(k)}$ and $\bbsym p_k$.
Under LICQ and strict complementarity, together with the positive-definite QP Hessian established above, the Jacobian of the KKT system with respect to $\bbsym w^{(k)}$ is nonsingular at the nominal solution. 
The implicit function theorem therefore guarantees the existence of a neighborhood of
$\bar{\bbsym p}_k$ in which the primal-dual solution is a continuously differentiable function of the QP parameters. 
Finally, since $\alpha_k$ is differentiable with respect to its trainable parameterization and $\bbsym d^{(k)}$ is locally differentiable with respect to the QP parameters, the SQP update \eqref{eq:sqp_update_trainable} is also locally continuously differentiable. 
This completes the proof.
\end{proof}

\section{Training Framework}
\label{sec:train}

In this section, we detail the proposed two-phase training framework. 
In Phase 1, the NN is trained to predict candidate solutions to the parametric MINLP.
In Phase 2, the trained NN is frozen and used to generate a number of candidate solutions. 
For each candidate, the discrete decisions are fixed, while the continuous solution serves as a warm start for a differentiable SQP refinement layer trained to improve feasibility and optimality.

\subsection{Training Neural Network Predictors}

In Phase 1, we train the NN predictor using a self-supervised loss combined with a supervised loss, if training labels are available.
First, let $\NSL$ and $\NSSL$ be the numbers of training instances for SL and SSL, respectively. 
We define the training dataset for supervised learning as 
${\mathcal{D}_{\mathrm{SL}} = \left\{ \left( \bbsym{\theta}_i, \bbsym{u}_i^{*}, \bbsym{z}_i^{*} \right) \right\}_{i=1}^{\NSL}}$, where $\bbsym{u}_i^{*}$, $\bbsym{z}_i^{*}$ denote the optimal continuous and discrete solutions generated from a solver, respectively.
Likewise, we define $\mathcal{D}_{\mathrm{SSL}} = \left\{ \bbsym{\theta}_j \right\}_{j=1}^{\NSSL}$ as the set of training samples for self-supervised learning, where the solution labels are not required.

For the integer variables, let $\Xi = \{\bbsym{\zeta}_1, \ldots, \bbsym{\zeta}_M\}$ denote the set of admissible values of the integer variable at each time step $t$. 
Let $\bbsym{\ell}_t = [\ell_{t,1}, \ldots, \ell_{t,M}]^\top \in \RR^M$ denote the NN output logits at time step $t$, with $\ell_{t,m}$ associated with $\bbsym{\zeta}_m$.
The corresponding categorical probability is:
\begin{equation}
p_{\phi}(\bbsym z_t = \bbsym{\zeta}_m \mid \bbsym{\theta}) = \frac{\exp(\ell_{t,m})} {\sum_{r=1}^{M}\exp(\ell_{t,r})},
\end{equation}
for $m=1,\ldots,M,\; t=0,\ldots,H-1$.
The supervised loss for the integer variables is given by the cross-entropy loss, averaged over the time steps and samples:
\begin{equation}
\begin{multlined}
\LLL_{\mathrm{int}} =
-\frac{1}{\NSL H} \sum_{i=1}^{\NSL} \sum_{t=0}^{H-1} \sum_{m=1}^{M} \mathbf{1}\!\left[ \bbsym z_{i,t}^{*} = \bbsym{\zeta}_m \right] \\
\qquad\qquad\times
\log p_{\phi}\!\left( \bbsym z_t = \bbsym{\zeta}_m \mid \bbsym{\theta}_i \right).
\end{multlined}
\end{equation}
For the continuous variables, we use a standard mean-squared error loss over the predicted input sequences:
\begin{equation}
\LLL_{\mathrm{cont}} = \frac{1}{\NSL}
\sum_{i=1}^{\NSL} \left\| \hat{\bbsym{u}}(\bbsym{\theta}_i, \hat{\bbsym z} (\bbsym{\theta}_i)) - \bbsym{u}_{i}^{*} \right\|_2^2.
\end{equation}
The supervised loss is then given by
\begin{equation}
\LLL_{\mathrm{SL}} = \lambda_{\mathrm{int}} \LLL_{\mathrm{int}} + \lambda_{\mathrm{cont}} \LLL_{\mathrm{cont}},
\end{equation}
where $\lambda_{\mathrm{int}},\lambda_{\mathrm{cont}}>0$ balance the integer and continuous prediction losses.

For computing the self-supervised loss, the integer prediction $\hat{\bbsym z}_j$ for each training instance $\theta_j$ is obtained from the NN logits using the straight-through Gumbel-Softmax estimator \cite{jang2017categorical}. 
In the forward pass, a hard one-hot mode assignment is used to evaluate the original objective and constraint functions, while during backpropagation, gradients are propagated through the corresponding continuous Gumbel-Softmax relaxation.
Given the predicted continuous variables $\hat{\bbsym x}_j = \bbsym x (\bbsym \theta_j)$ and the integer variables $\hat{\bbsym z}_j = \bbsym z (\bbsym \theta_j)$, we define the self-supervised loss as follows,
\begin{equation}
\begin{aligned}
\LLL_{\mathrm{SSL}} &= \frac{1}{\NSSL} \sum_{j = 1}^{N_\mathrm{SSL}} \!\!
\Big[  f \left( \hat{\bbsym x}_j, \hat{\bbsym z}_j; \bbsym{\theta}_j \right) 
+\rho_h \left\| \bbsym h\left( \hat{\bbsym x}_j, \hat{\bbsym z}_j; \bbsym{\theta}_j \right) \right\|_2^2 \\
& \qquad\qquad\qquad  
+ \rho_g \left\| \left[ \bbsym g\left( \hat{\bbsym x}_j, \hat{\bbsym z}_j; \bbsym{\theta}_j \right) \right]_+ \right\|_2^2 \Big],
\end{aligned}
\end{equation}
where $\rho_h,\rho_g>0$ are penalty weights.

The training loss for Phase 1 linearly combines the supervised and self-supervised losses as follows,
\begin{equation}
\LLL_{\mathrm{P1}} = (1-\beta) \LLL_{\mathrm{SL}} + \beta \LLL_{\mathrm{SSL}},
\end{equation}
where $\beta \in [0,1]$ balances the advantages of supervised and self-supervised terms.

\subsection{Training SQP Refinement Layer}

Phase 1 provides a fast prediction of both the integer and continuous solutions. 
However, due to prediction errors, the integer solution associated with the highest logit may be infeasible or suboptimal, while alternative candidates can yield lower constraint violation or objective value.
Therefore, we instead consider the $B$ most promising candidates from the NN.
The main goal of Phase 2 is to train the SQP layer to refine the $B$ continuous warm starts over a fixed number of iterations to improve optimality and feasibility.

First, given an instance $\bbsym{\theta}$ of the problem parameters and the NN trained from Phase 1,  we construct $B$ candidate integer solutions, together with their associated continuous warm starts $\{ \big( \hat{\bbsym{x}}^{(b)}, \hat{\bbsym{z}}^{(b)} \big) \}_{b=1}^{B}$.
For each candidate, the integer solution $\hat{\bbsym{z}}^{(b)}$ is fixed, reducing
the original mixed-integer problem to the continuous nonlinear program
\begin{equation}
\begin{aligned}
\minimize_{\bbsym{x}} \quad &
f \left(\bbsym{x}, \hat{\bbsym{z}}^{(b)}; \bbsym{\theta}\right) \\
\subjectto \quad & \bbsym{h}\left(\bbsym{x}, \hat{\bbsym{z}}^{(b)}; \bbsym{\theta}\right) = \bbsym{0}, \\
& \bbsym{g} \left(\bbsym{x}, \hat{\bbsym{z}}^{(b)}; \bbsym{\theta} \right) \le \bbsym{0}.
\end{aligned}
\end{equation}
Starting from $\hat{\bbsym{x}}^{(b)}$ as the initial continuous solution $\bbsym{x}^{(0,b)}$, we perform refinement through $K$ differentiable QP layers to obtain $\bbsym{x}^{(K,b)}$. 
Note that here we use $b=1,\ldots,B$ to index the candidate solutions generated by the frozen predictor, and $k=1,\ldots,K$ to index the QP layers in SQP refinement.
To evaluate the quality of each SQP iteration, we define a feasibility-aware merit function that combines the objective with penalties on constraint violations of the original MINLP problem as follows:
\begin{equation}
\begin{aligned}
\Phi(\bbsym{x},\bbsym{z};\bbsym{\theta}) &=
f (\bbsym{x},\bbsym{z};\bbsym{\theta}) + \rho_h \left\| \bbsym{h}(\bbsym{x},\bbsym{z};\bbsym{\theta}) \right\|_2^2 \\
& + \rho_g \left\| \left[ \bbsym{g}(\bbsym{x},\bbsym{z};\bbsym{\theta}) \right]_+ \right\|_2^2 +
\rho_{\mathrm{tail}}\, \VVV(\bbsym{x},\bbsym{z};\bbsym{\theta}),
\end{aligned}
\end{equation}
where $\rho_{\mathrm{tail}}>0$ is a penalty weight and the term $\rho_{\mathrm{tail}} \, \VVV(\bbsym{x}, \bbsym{z}; \bbsym{\theta})$ penalizes the largest constraint violations, defined by a smooth approximation of the maximum over the equality and inequality residuals.

The SQP parameters are therefore trained by minimizing the merit function over the $B$ candidates. 
For each $b$, we define the beam score as:
\begin{equation}
S^{(b)} = \Phi\!\left( \bbsym{x}^{(K,b)}, \bbsym{z}^{(b)}; \bbsym{\theta} \right) 
\!+\! \sum_{k=1}^{K-1} \lambda_k\, \Phi \! \left( \bbsym{x}^{(k,b)}, \bbsym{z}^{(b)}; \bbsym{\theta} \right),
\end{equation}
where $\lambda_k \ge 0$ weights the contributions of the SQP layers.
The training loss for Phase 2 is then computed by averaging over $N_{\mathrm{P2}}$ problem instances as follows:
\begin{equation}
\LLL_{\mathrm{P2}} = \frac{1}{N_{\mathrm{P2}}}
\sum_{n=1}^{N_{\mathrm{P2}}} \sum_{b=1}^{B} w_n^{(b)} S_n^{(b)},
\end{equation}
where $w_n^{(b)}$ denotes the differentiable soft-selection weight assigned to candidate $b$ for problem instance $n$, defined as:
\begin{equation}
w_n^{(b)} =
\frac{\exp\!\left(-S_n^{(b)}/(\tau s_n)\right)}{\sum_{r=1}^{B} \exp\!\left(-S_n^{(r)}/(\tau s_n)\right)},
\end{equation}
where $\tau > 0$ controls the sharpness of the selection and
$s_n>0$ normalizes the merit scale for each problem instance.
This soft selection focuses the gradient on promising candidates while still allowing multiple near-optimal candidates to contribute to training.

\section{Numerical Examples}
\label{sec:sim}

In this section, we evaluate the performance of the proposed framework on a numerical example for contact scheduling and trajectory optimization in planar box pushing.

\subsection{Contact Scheduling for Box Pushing}

We formulate an MINLP for planar trajectory and contact scheduling in which a pusher drives a rectangular object to a desired pose through contact and free-motion modes \cite{graesdal2024tight}. 
We retain the quasi-static pushing model, including the ellipsoidal limit-surface approximation, frictional contact constraints, and force-motion coupling, but replace the graph-of-convex-sets and semidefinite-relaxation formulation in \cite{graesdal2024tight} with an MINLP formulation. 
Binary variables select face-contact or collision-free modes at each time step, with sticking contact enforced in active contact modes and contact-face transitions required to pass through free motion.

Let $\bbsym p_k \in \PPP$ and $\bbsym s_k \in \SSS_p$ denote the positions of the object and pusher, respectively, where $\PPP$ and $\SSS_p$ denote their admissible position sets.
Let $\theta_k$ denote the object orientation, $\bbsym u_k \in \RR^2$ the pusher velocity, $\bbsym f_k \in \RR^2$ the contact force expressed in the object frame, and $\tau_k$ the corresponding contact torque.
We consider the following bound constraints on the states and inputs \begin{equation}
\theta_k \in [\underline{\theta},\overline{\theta}], \;
\underline u_j \le u_{x,k}, u_{y,k} \le \overline u_j.
\end{equation}
The object and pusher evolve according to:
\begin{equation}
\begin{split}
\bbsym p_{k+1} &= \bbsym p_k + \Delta t\,R(\theta_k)\bbsym f_k/f_{\max}^2,\\
\theta_{k+1} &= \theta_k + \Delta t\,\tau_k/\tau_{\max}^2, \\
\bbsym s_{k+1} &= \bbsym s_k + \Delta t\,\bbsym u_k,
\end{split}
\end{equation}
where $f_{\max}$ and $\tau_{\max}$ denote the characteristic force and torque scales of the quasi-static model, and $R(\theta_k)$ is the rotation matrix associated with $\theta_k$.
Let $\bbsym q_k$ denote the pusher position expressed in the object frame:
\begin{equation}
\bbsym q_k = R(\theta_k)^\top (\bbsym s_k-\bbsym p_k).
\end{equation}
Denote the contact-mode and free-space-mode sets as $\MMM_c=\{1,\ldots,4\}$, $\MMM_f=\{5,\ldots,8\}$, respectively.
Modes in $\MMM_c$ correspond to contact with the four faces, while modes in $\MMM_f$ represent the corresponding collision-free regions.
We impose a constraint that exactly one mode is selected at each time step as:
\begin{equation}
\sum_{m\in\MMM} z_{m,k}=1, \; z_{m,k}\in\{0,1\}, \; m \in \MMM,
\label{eq:box_mode}
\end{equation}
where $\MMM = \MMM_c\cup\MMM_f$.
Define the contact indicator ${c_k=\sum_{i=1}^{4} z_{i,k}}$, then the force and torque satisfy:
\begin{equation}
-\bar f c_k\le f_{j,k}\le\bar f c_k,
\; -\bar \tau c_k \le \tau_k \le \bar \tau c_k.
\label{eq:force_activation}
\end{equation}
Thus, when a free-motion mode is selected, $c_k=0$ and both the contact force and torque vanish.
For a rectangular object with half-width $a$ and half-height $b$, define
$\bbsym v_1 = [-a,b]^\top$,
$\bbsym v_2 = [a,b]^\top$,
$\bbsym v_3 = [a,-b]^\top$,
$\bbsym v_4 = [-a,-b]^\top$.
We define the corresponding inward normals and tangents as $\bbsym n_i$ and $\bbsym t_i$ for $i\in\{1,\dots,4\}$, respectively.
For each contact face, we define $\delta_{i,k} = M(1 - z_{i,k})$ and $\bbsym q^c_{i,k} = \bbsym q_k + r\bbsym n_i$,
where $M$ is the big-$M$ constant and $r$ is the pusher radius. 
The normal and tangential components of the contact force are
\begin{equation}
f^n_{i,k}=\bbsym n_i^\top\bbsym f_k,
\;
f^t_{i,k}=\bbsym t_i^\top\bbsym f_k.
\end{equation}
The finite extent of face $i$ is represented by the minimum and maximum tangential coordinates of its two endpoints:
\begin{equation}
\ell_i = \min_{j\in\{i,i+1\}} \bbsym t_i^\top\bbsym v_j, \; u_i = \max_{j\in\{i,i+1\}} \bbsym t_i^\top\bbsym v_j,
\end{equation}
where the vertex indices are interpreted cyclically, with
$\bbsym v_5 \equiv \bbsym v_1$.
Next, we present the contact constraints that are grouped into contact geometry, contact force, and torque consistency and sticking conditions.
\paragraph{Contact geometry}
The candidate contact point must lie on the supporting line of face $i$ and within its finite tangential extent as:
\begin{align}
-\delta_{i,k} &\le \bbsym n_i^\top (\bbsym q^c_{i,k}-\bbsym v_i) \le \delta_{i,k}, \label{eq:contact_plane} \\
\ell_i-\delta_{i,k} &\le \bbsym t_i^\top\bbsym q^c_{i,k} \le u_i+\delta_{i,k}.
\label{eq:contact_extent}
\end{align}

\paragraph{Contact force}
The contact force must be compressive along the inward normal and satisfy the Coulomb friction constraint as follows:
\begin{align}
-\delta_{i,k} &\le f^n_{i,k} \le \bar f+\delta_{i,k}, \label{eq:contact_normal} \\
-\mu f^n_{i,k}-\delta_{i,k} &\le f^t_{i,k} \le \mu f^n_{i,k}+\delta_{i,k}. \label{eq:contact_friction}
\end{align}

\paragraph{Torque consistency and sticking contact}
If contact mode is active, the contact torque must be equal to the moment generated by the contact force about the object center, leading to the following constraint:
\begin{equation}
-\delta_{i,k} \le \tau_k - \left( q^c_{i,k,x}f_{y,k} - q^c_{i,k,y}f_{x,k} \right) \le \delta_{i,k}.
\label{eq:contact_torque}
\end{equation}
In addition, sticking contact is imposed through:
\begin{equation}
-\delta_{i,k}\mathbf 1 \le \bbsym q_{k+1}-\bbsym q_k \le \delta_{i,k}\mathbf 1.
\label{eq:contact_sticking}
\end{equation}
If $z_{i,k}=1$ then $\delta_{i,k}=0$, and these constraints place the pusher on the selected face, enforce a unilateral frictional contact force and mechanically consistent torque, and maintain a fixed body-frame contact location.
For the free-space motion, for each collision-free region associated with the face ${i \in \MMM_c}$, let
$(\bbsym A_{i,\ell},b_{i,\ell})$, $\ell\in\{0,1,2\}$, denote its normalized bounding hyperplanes, and define ${d_0=r}$, $d_1=d_2=0$.
The corresponding free-space mode is indexed by $i+4$, and if this mode is active, both endpoints of the pusher must satisfy
\begin{equation}
\bbsym A_{i,\ell}^\top\bbsym q_e - b_{i,\ell} \ge d_\ell-M(1-z_{i+4,k}),
\; e\in\{k,k+1\}.
\label{eq:free_region}
\end{equation}
Because each region is convex, constraining both endpoints to the same region ensures that the straight-line pusher motion over the interval remains collision-free.
Finally, direct transitions between different contact faces are prohibited:
\begin{equation}
z_{i,k}+z_{j,k+1} \le 1, \; i\neq j,\; i,j \in \{1, \dots, 4\},
\label{eq:contact_transition}
\end{equation}
so that changing contact faces requires at least one intervening free-motion mode. 
Mode-switching variables satisfy
\begin{equation}
\sigma_k \ge z_{m,k}-z_{m,k+1}, \; \forall m = 1 \ldots, 8,
\label{eq:switch_penalty}
\end{equation}
and are penalized in the objective.
Given a desired object pose $(\bbsym p_g,\theta_g)$, the objective for this task is
\begin{equation}
\begin{multlined}
J = w_{p,f}\norm{\bbsym p_H-\bbsym p_g}_2^2 + w_{\theta,f}\left[1-\cos(\theta_H-\theta_g)\right] \\ 
+ \sum_{k=0}^{H-1} \left( w_p\norm{\bbsym p_k-\bbsym p_g}_2^2 + w_u\norm{\bbsym u_k}_2^2 + w_f\norm{\bbsym f_k}_2^2 \right) \\
+ w_{\sigma}\sum_{k=0}^{H-2} \sigma_k,
\end{multlined}
\label{eq:box_push_obj}
\end{equation}
where $w_{p,f},\, w_{\theta,f},\, w_p,\, w_u,\, w_f$, and $w_{\sigma}$ are nonnegative weights. 
The resulting problem is an MINLP due to the nonlinear rotation $R(\theta_k)$ and the bilinear force-moment relation in \eqref{eq:contact_torque}, alongside the binary variables used to select contact and free-space modes.

\begin{figure}[t]
\centering
\begin{subfigure}[t]{0.49\linewidth}
\centering
\includegraphics[
height=4.8cm,
]{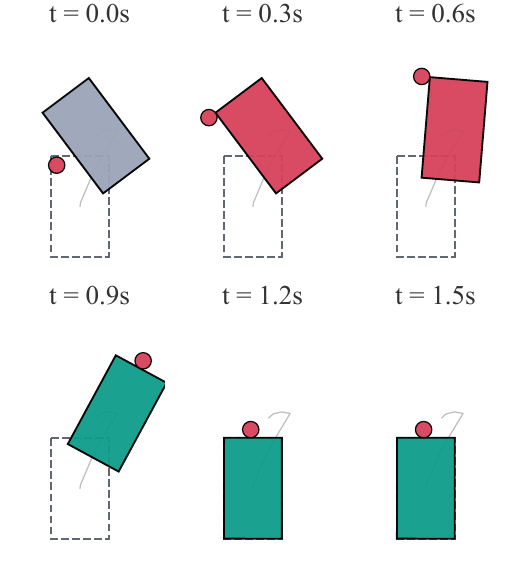}
\label{fig:ex2-snapshot-a}
\end{subfigure}
\hfill
\begin{subfigure}[t]{0.49\linewidth}
\centering
\includegraphics[
height=4.8cm,
]{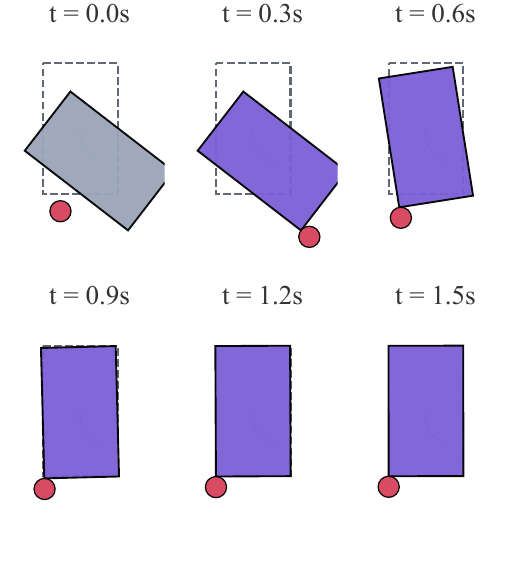}
\label{fig:ex2-snapshot-b}
\end{subfigure}
\vspace{-4mm}
\caption{Snapshots of the planar box-pushing trajectories. Different colors indicate contact with different faces of the box.}
\label{fig:ex2-snapshot}
\vspace{-4mm}
\end{figure}

\subsection{Results and Discussions}

We generate a labeled dataset consisting of $300$ problem instances using the \texttt{GUROBI} optimizer \cite{gurobi} to train the NN using a combination of supervised and self-supervised losses in Phase 1.
Note that obtaining ground-truth labeled solutions is computationally expensive, which limits the number of available training data points. 
As a result, this example is particularly useful for evaluating the proposed method under limited data.
We additionally generate a separate set of $100$ scenarios to evaluate the trained model.
Video visualizations of the results and implementation details are available on the supplemental website:
\url{https://sites.google.com/view/l2o-minlp}.
In Figure~\ref{fig:ex2-snapshot}, we show two representative planar-pushing trajectories, where the pusher switches between box faces to steer the object toward the target pose. 
In both cases, the box successfully reaches the desired configuration using the trained L2O model with $B=8$ and $K=3$.

Next, we focus on evaluating the effect of the SQP refinement using different numbers of SQP layers and solution candidates.
We compute and report the following metrics: (1) the optimality gap with respect to the reference solution obtained using the \texttt{GUROBI} optimizer and (2) the worst-case constraint violation, as shown in Figure~\ref{fig:ex2_boxplot}.
The figure shows that the SQP refinement layer improves both optimality and feasibility relative to the NN predictor, with constraint violations generally decreasing as the number of SQP layers and solution candidates increases.
At $B=8$ and $K=3$, the SQP refinement reduces the constraint violation by approximately $10$ times compared with the NN-only prediction.
Overall, the results consistently indicate that these improvements can be achieved with only a few SQP layers and solution candidates, which supports the use of fixed-depth refinement for real-time implementation.

\begin{figure}[t]
\centering
\includegraphics[width=\linewidth]{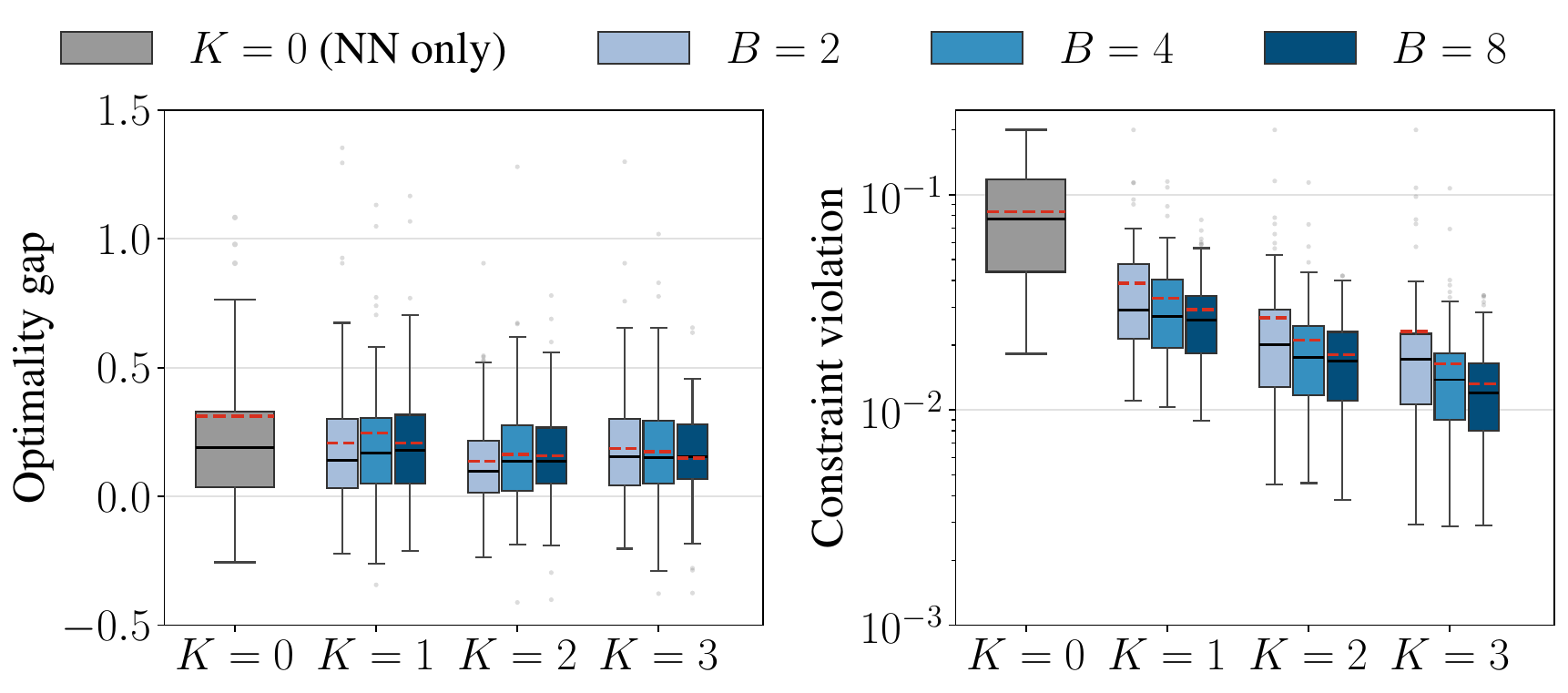}
\caption{A boxplot showing the optimality gap (left) and worst-case constraint violation in log scale (right) over $100$ held-out validation scenarios for $K \in \{0,1,2,3\}$ and $B \in \{2,4,8\}$. Here, $K=0$ denotes the NN predictor without SQP refinement. The black solid and red dashed lines indicate the median and mean, respectively.}
\label{fig:ex2_boxplot}
\vspace{-4mm}
\end{figure}

Next, we show in Table~\ref{tab:planar-pushing-timing} the online computation time of the proposed L2O framework with different $K$ and $B$, together with the baseline \texttt{GUROBI} solver. 
As can be seen from the table, the computational benefits of L2O are evident in this planar-pushing example, where the proposed method requires less than \SI{1}{s} even with $B=8$ and $K=3$, while \texttt{GUROBI} reaches the \SI{400}{s} time limit on most instances.
For this example, \texttt{GUROBI} faces a challenging optimization problem because the discrete contact-mode decisions are tightly coupled with nonlinear force and torque variables through nonconvex contact dynamics, resulting in weaker relaxations and a more difficult branch-and-bound search.
Overall, these runtime comparison results demonstrate a trade-off between computational cost and solution quality as the numbers of candidates and SQP layers increase.

\textbf{Limitations:} 
Although the proposed framework demonstrates promising performance, we identify several limitations. 
First, the training and inference costs increase as the number of SQP layers grows. 
Second, scalability to high-dimensional problems remains challenging, as generating sufficient high-quality training data for such problems may be difficult, and both the forward and backward passes through the QP layers can become computationally expensive. 
Addressing these challenges will be an important direction for future work in L2O for MINLP.

\setlength{\tabcolsep}{2pt}
\begin{table}[t]
\centering
\caption{Runtime comparison (in seconds) for planar box pushing for different
numbers of candidates $B$ and SQP layers $K$.}
\label{tab:planar-pushing-timing}
\small
\begin{tabular*}{\linewidth}{@{\extracolsep{\fill}}lccrrr@{}}
\toprule
Method & $B$ & $K$ & Mean (s) & Min (s) & Max (s) \\
\midrule
\texttt{GUROBI} & -- & -- & $395.29$ & $7.37$ & $400.87$ \\
\midrule
\multirow{9}{*}{L2O}
& $2$ & $1$ & $0.07$ & $0.07$ & $0.08$ \\
& $2$ & $2$ & $0.14$ & $0.14$ & $0.14$ \\
& $2$ & $3$ & $0.23$ & $0.23$ & $0.23$ \\
\cmidrule{2-6}
& $4$ & $1$ & $0.15$ & $0.14$ & $0.15$ \\
& $4$ & $2$ & $0.28$ & $0.28$ & $0.30$ \\
& $4$ & $3$ & $0.47$ & $0.47$ & $0.48$ \\
\cmidrule{2-6}
& $8$ & $1$ & $0.29$ & $0.28$ & $0.30$ \\
& $8$ & $2$ & $0.60$ & $0.59$ & $0.61$ \\
& $8$ & $3$ & $0.92$ & $0.91$ & $0.94$ \\
\bottomrule
\end{tabular*}
\vspace{-4mm}
\end{table}

\section{Conclusions}
\label{sec:conc}

This paper presented an L2O framework for solving parametric MINLPs that combines NN prediction of integer solutions and continuous initializations with trainable SQP refinement.
By generating multiple candidate integer solutions and refining their associated continuous warm starts with a trainable SQP layer, the framework leverages the underlying problem structure to improve feasibility and optimality, while reducing the computational effort required to solve the MINLPs.
The performance was demonstrated on contact scheduling and trajectory optimization for planar box pushing, where the proposed framework improves solution quality relative to using the NN predictor solely, and achieves significant reductions in computation time compared with solving the MINLP using a state-of-the-art numerical solver.

\section{GenAI Disclosure}

ChatGPT (OpenAI) \cite{chatgpt2026} was used to improve the syntax and grammar of several paragraphs in the manuscript.

\bibliographystyle{IEEEtran}
\bibliography{IEEEabrv,references}

\end{document}